\documentclass[12pt]{amsart}

\usepackage{amssymb,latexsym,amsmath,amsthm,graphicx}
\usepackage{epstopdf}
\usepackage{wrapfig}

\newtheorem{q}{Question}
\newtheorem{thm}{Theorem}
\newtheorem{lemma}[thm]{Lemma}
\newtheorem{cor}[thm]{Corollary}

\theoremstyle{remark}

\usepackage{xcolor}

\newcommand{\beq}{\begin{equation}}
\newcommand{\eeq}{\end{equation}}

\newcommand{\RR}{\mathbb{R}}
\newcommand{\EE}{\mathbb{E}}

\newcommand{\CC}{\mathbb{C}}

\newcommand{\ol}{\overline}

\begin{document}
\author[A. Lerario, E. Lundberg]{Antonio Lerario and Erik Lundberg}
\title[Random harmonic polynomials]{On the zeros of random harmonic polynomials: the truncated model}

\begin{abstract}Motivated by Wilmshurst's conjecture and more recent work of W. Li and A. Wei \cite{LiWei2009}, 
we determine asymptotics for the number of zeros of random harmonic polynomials sampled from the \emph{truncated model}, recently proposed by J. Hauenstein, D. Mehta, and the authors  \cite{HLLM}.
Our results confirm (and sharpen) their
$(3/2)-$powerlaw conjecture \cite{HLLM} that had been formulated on the basis of computer experiments;
this outcome is in contrast with that of the model studied in \cite{LiWei2009}.
For the truncated model we also observe a phase-transition in the complex plane for the Kac-Rice density.
\end{abstract}

\maketitle

\section{Introduction}

A harmonic polynomial is a complex-valued harmonic function given by:
\beq \label{eq:harmonic}F(z)=p(z)+\overline{q(z)},\eeq
where $p$ and $q$ are polynomials of degree $n$ and $m$ (respectively).
Let $N_F$ denote the number of zeros of $F$,
that is, points $z\in\CC$ such that $F(z) = 0$.  

For $n>m$, we have the following bounds:
$$ n \leq N_F \leq n^2.$$

The lower bound is based on the generalized argument principle and is sharp for each $m$ and $n$.
The upper bound follows from applying Bezout's theorem to the real and imaginary parts of $F(z)=0$
after noticing that the zeros are isolated, which was shown by Wilmshurst \cite{Wilm98}.

\subsection{Wilmshurst's conjecture}

Wilmshurst  made the conjecture that
the Bezout bound can be improved to a function that is linear in $n$
for each fixed $m$, namely:
\beq \label{eq:wilm} N_F\leq 3n-2+m(m-1)\qquad \textrm{(Wilmshurst's conjecture)}\eeq
This conjecture is stated in \cite[Remark 2]{Wilm98} (see also \cite{Sh2002} and \cite{BL2010}).

For $m=n-1$, the upper bound follows from Wilmshurst's theorem \cite{Wilm98},
and examples were also given in~\cite{Wilm98} showing that this bound is sharp (shown independently in \cite{BHS1995}).
For $m=1$, the upper bound was shown by Khavinson and Swiatek \cite{K-S} using anti-holomorphic dynamics.
A proof of the Crofoot-Sarason conjecture given in \cite{G2008} (cf. \cite{BL2004}) established that this bound is sharp.
Counterexamples to the case $m=n-3$ were established analytically in \cite{LLL},
and counterexamples for a broad range of (finitely many) $m$ and $n$ were established
in \cite{HLLM} using certified numerics.
On the other hand, we still expect, in the spirit of (\ref{eq:wilm}),
that $N_F$ satisfies an upper bound that is linear in $n$ for $m$ fixed;
for instance, with S-Y. Lee, the authors conjectured in \cite[Introduction]{LLL} that $N_F\leq 2 m(n-1) + n$.

\subsection{A probabilistic version of the problem}

Given the high variability of the number of zeros $N_F$,
it is natural to ask the following.

\begin{q}\label{q:LW}
What is the expectation $\EE N_F$ of the number of zeros of a random harmonic polynomial?
\end{q}

This question was asked and answered by W. Li and A. Wei in \cite{LiWei2009},
in the case when $p$ and~$q$ are independently sampled from the complex Kostlan ensemble:
\begin{equation} \label{eq:LW}
p(z) = \sum_{k=0}^{n} a_k  z^k, \quad q(z) = \sum_{k=0}^{m} b_k z^k,
\end{equation}
where $a_k$ and $b_k$ are independent centered complex Gaussians with $\EE a_j \ol{a_k} = \delta_{jk} \binom{n}{j}$ and $\EE b_j \ol{b_k} = \delta_{jk} \binom{m}{k}$.

The choices of $p$ and $q$ in (\ref{eq:LW}) lead to the following asymptotics (as $n \rightarrow \infty$):
\beq\label{eq:dichotomy} 
\EE N_F \sim
\left\{ \begin{array}{cc}
  \frac{\pi}{4} n^{3/2}, \quad \text{when } m = n, \quad \quad \quad \quad \quad \quad \quad \quad \quad \quad \\
 n , \quad \text{when } m = \alpha n + o(n) \text{ with } 0 < \alpha < 1,
\end{array} \right.
\eeq
Notice that when $m=\alpha n$ the average number of zeros is asymptotically the \emph{fewest possible}.
This seems to suggest that, on average, an even stronger form of Wilmshurst's conjecture (\ref{eq:wilm}) holds.
However, caution is needed here,
and the dichotomy in (\ref{eq:dichotomy})
dissolves after choosing a definition of ``random''
in which the coefficients of $p$ and $q$ are more comparable in modulus
(see Theorem \ref{thm:asymp} below).

In the model (\ref{eq:LW}), where the coefficients of $p$ are asymptotically much larger in modulus than $q$ when $m=\alpha n$,
 $F$ tends to resemble an \emph{analytic} polynomial and asymptotically obeys the fundamental theorem of algebra.
In order to make $q$ more comparable to $p$, 
an alternative model (referred to as the ``truncated model'') was proposed in \cite{HLLM}
where the variances $\binom{m}{k}$ were 
replaced by $\binom{n}{k}$ in the definition (\ref{eq:LW}) of $q$,
while still choosing $m$ as the upper limit in the summation (see definition (\ref{eq:trunc}) below).
For the truncated model, 
computer experiments performed in \cite{HLLM} 
led to a conjecture that the expectation $\EE N_F$ has a $(3/2)-$powerlaw growth whenever $m = \alpha n$
for all $0<\alpha<1$.
Here, we prove (and sharpen) this conjecture, see Theorem \ref{thm:asymp} below.

Note that we do not consider here the case $m=0$ of random complex analytic polynomials,
where we would have $N_F = n$ almost surely (by the fundamental theorem of algebra).
Yet, it is still interesting in that case to study the \emph{location} of zeros; 
we refer the reader to Edelman and Kostlan's paper \cite[Sec. 8]{EK1995}
and to the recent work of Zeitouni and Zelditch \cite{ZZ}
establishing a large deviation principle for the location of the zeros
of a random analytic polynomial.

\subsection{Asymptotics for the truncated model}

We revisit Question \ref{q:LW} 
while sampling $F(z) = p(z) + \ol{q(z)}$ randomly from the truncated model, i.e.,
\begin{equation} \label{eq:trunc}
p(z) = \sum_{k=0}^{n} a_k z^k, \quad q(z) = \sum_{k=0}^{m} b_k  z^k,
\end{equation}
where $a_k$ and $b_k$ are independent centered complex Gaussians with $\EE a_j \ol{a_k} = \delta_{jk} \binom{n}{j}$ and $\EE b_j \ol{b_k} = \delta_{jk} \binom{n}{k}$.

\begin{thm}\label{thm:asymp}
Let $F(z) = p_n(z) + \overline{q_m(z)}$ 
be a random polynomial from the truncated model.
For $m = \alpha n$ with $0 < \alpha < 1$,
the expectation $\EE N_F$ of the number of zeros of $F(z)$
satisfies the following asymptotic (as $n \rightarrow \infty$)
$$\EE N_F \sim c_\alpha n^{3/2},$$
where $c_\alpha$ is given by 
\begin{equation}\label{eq:const}
c_\alpha = \frac{1}{2} \left( \arctan \left( \sqrt{\frac{\alpha}{1-\alpha} } \right) - \sqrt{\alpha(1-\alpha)} \right).
\end{equation}
On the other hand, when $n \rightarrow \infty$ with $m$ fixed we have $\EE N_F \sim n$.
\end{thm}

Our methods can be used to describe asymptotics for the Kac-Rice density
(providing the expected number of zeros over a prescribed region).
We notice a phase-transition in this pointwise asymptotic, 
and the leading contribution $c_\alpha n^{3/2}$ is completely accounted for
by zeros that are located within a critical distance from the origin, see Section \ref{sec:crit}.

Note that as $\alpha \rightarrow 1$,
$c_\alpha \rightarrow \pi / 4$, in agreement with \cite[Thm. 1.1]{LiWei2009}.

\medskip

An interesting aspect of harmonic polynomials
is that, unlike analytic polynomials, the function $F(z) = p(z) + \ol{q(z)}$ can reverse orientation.
The orientation of $F$ can be determined by the sign of the Jacobian determinant $J_F(z) = |p'(z)|^2 -  |q'(z)|^2$.
Let $N_+$ denote the number of zeros for which $F$ is orientation-preserving (i.e., $J_F < 0$) and $N_-$ denote the number of zeros where $F$ is orientation-reversing ($J_F > 0$).

Using a standard application of the generalized argument principle,
we then notice the following corollary of Theorem \ref{thm:asymp},
showing that orientation-reversing zeros are asymptotically as common as orientation-preserving ones.

\begin{cor}\label{cor:orientation}
For $m = \alpha n$ with $0 < \alpha < 1$, we have $\EE N_+ \sim \EE N_- \sim \frac{c_\alpha}{2} n^{3/2}$.
\end{cor}

\begin{proof}

Almost surely we have $N_F = N_+ + N_-$ (the presence of singular zeros is a probability zero event).
By topological degree theory (or the generalized argument principle \cite{DHL})
the difference $N_+ - N_-$ is given by the winding number of $F$ along a sufficiently large circle.
Moreover, since the $z^n$ term dominates, the winding number is $n$, and
so we have $N_+ = N_- + n$. 
Theorem \ref{thm:asymp} then implies that $\EE N_+ \sim \EE N_- \sim \frac{c_\alpha}{2} n^{3/2}$.
\end{proof}

The coexistence of many zeros of opposite orientation
suggests that the Jacobian of $F$ changes sign wildly
throughout the complex plane 
(or otherwise that there is a high level of ``condensation'' of zeros 
into regions of common orientation).
Taking this point into consideration, we conclude the introduction by posing the problem
of investigating the topology of the orientation-reversing set $\Omega_- := \{z \in \CC: |p'(z)| < |q'(z)| \}$.
It follows from applying the maximum principle to the harmonic function
$\log |p'(z)| - \log |q'(z)|$
that each connected component of $\Omega_-$ contains at least one critical point of $p$.
This implies that $\Omega_-$ has at most $n-1$ connected components.
What can be said about the \emph{average} number of components of $\Omega_-$?
\begin{figure}
\centering
\includegraphics[scale=.5]{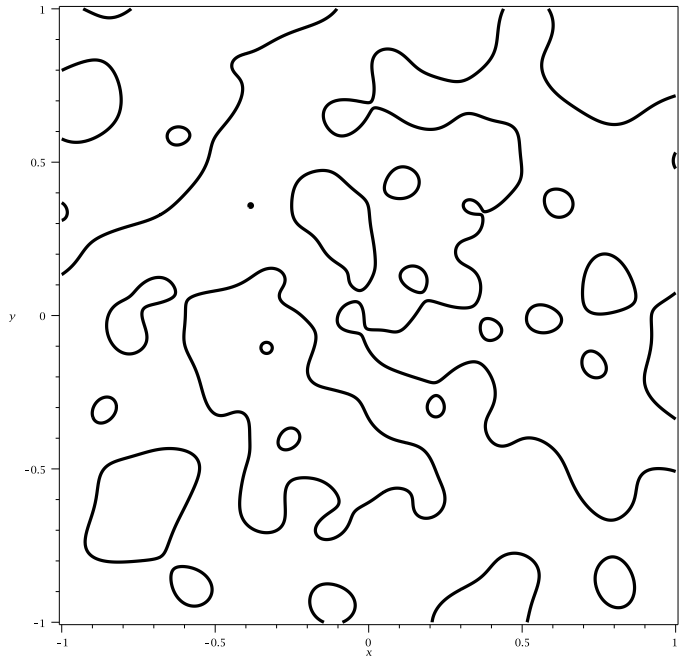}
    \caption{A portion of a random critical lemniscate (the critical set of a random harmonic polynomial) with $m=n=100$.
    Plotted in the region $\{z \in \CC: |\Re z | < 1, |\Im z| <1\}$.
    For $m=n$ the truncated model coincides with the Li-Wei model.}
    \label{fig:lemniscate}
\end{figure}
The critical set (the boundary of $\Omega_-$) is depicted in figure \ref{fig:lemniscate}
for a random sample with $m=n=100$.
Note that the critical set is a random rational lemniscate,
\begin{equation}\label{eq:randcritlem}
 \left\{ z \in \CC : \left| \frac{p'(z)}{q'(z)} \right| = 1 \right\},
\end{equation}
similar to the random lemniscates studied recently by the authors \cite{LLrandlem};
the only difference is that in the model studied in \cite[Sec. 1.2]{LLrandlem},
the numerator $p$ and denominator $q$ of the rational function appear without differentiation.
Based on the results in \cite{LLrandlem} we conjecture that when $m=n\rightarrow \infty$ the
average number of connected components of the random critical lemniscate (\ref{eq:randcritlem}) grows linearly
(the maximum rate possible).


\subsection{Outline}
In Section \ref{sec:KacRice}, 
we provide an exact formula for the average number of zeros for the truncated model.
This is derived from a slight modification of \cite{LiWei2009}.
The asymptotics stated in Theorem \ref{thm:asymp} are proved in Section \ref{sec:asymp}.
The proof uses the dominated convergence theorem after factoring out $n^{3/2}$.
Establishing a dominating function requires several elementary estimates,
and determining the pointwise limit of the integrand requires 
asymptotics for a truncated binomial sum.
Such asymptotics are provided in Lemma \ref{lemma:Laplace},
and the proof of Lemma \ref{lemma:Laplace} is given in the separate Section \ref{sec:lemmaproof}.
The proof uses both forms of Laplace's asymptotic method \cite[Sec. 3.3, 3.4]{Miller}: 
namely the case of an interior maximum (saddle-point) as well as the case of an end-point maximum.
The presence of both cases is responsible for the phase-transition in the Kac-Rice density mentioned above
(see also Section \ref{sec:crit}).

\smallskip

\noindent {\bf Acknowledgement.}
We wish to thank Seung-Yeop Lee for inspiring discussions and insightful suggestions and the anonymous referees for their helpful remarks.

\section{An exact formula for $\EE N_f$}\label{sec:KacRice}

Let $P_{m,n}(x) := \sum_{k=0}^m \binom{n}{k} x^k$ denote the binomial expansion of
$(1+x)^n$ truncated at degree $m$.

\begin{thm}\label{thm:exact}
The expectation $\EE N_F(T)$ of the number of zeros of $F_{n,m}(z) = p_n(z) + \overline{q_m(z)}$
on a domain $T \subset \CC$ is given by:
\begin{equation}\label{eq:KR}
\EE N_F(T) = \frac{1}{\pi} \int_{T} \frac{1}{|z|^2} \frac{R_1^2 + R_2^2 - 2 R_{12}^2}{R_3^2 \sqrt{(R_1 + R_2)^2 - 4 R_{12}^2}} dA(z),
\end{equation}
where $dA(z)$ denotes the Lebesgue measure on the plane, and
\begin{align*}
R_{12} &=n^2|z|^4(1+|z|^2)^{n-1} P_{m-1,n-1}(|z|^2), \\
R_3 &= (1+|z|^2)^n +P_{m,n}(|z|^2) , \\
R_1 &=  R_3 (n^2 |z|^4 + n |z|^2)(1+|z|^2)^{n-2} - n^2|z|^4(1+|z|^2)^{2n-2} , \\
R_2 &= R_3 [ n^2 |z|^4 P_{m-2,n-2}(|z|^2) + n|z|^2 P_{m-1,n-2}(|z|^2) ] - n^2|z|^4 [P_{m-1,n-1}(|z|^2)]^2  .
\end{align*}
\end{thm}

\noindent Note: The analogous statements contained in \cite[Thm. 1.1, Thm. 4.1]{LiWei2009}
contain a little ambiguity. In fact the authors use the Kac-Rice formula for the harmonic function \emph{already} in polar coordinates, 
thus viewing it as a random field defined over $[0, 2\pi)\times (0, \infty)$ and with values in $\RR^2$. 
In particular  \cite[Equation (1.1)]{LiWei2009} should either be modified with $|z|^2$ instead of $|z|$ (and $d\sigma(z)$ is still the Lebesgue measure on the complex plane) 
or the integration should be performed over the image of $T\subset \CC$ under the polar change of coordinates 
(and in this case $|z|=\rho$). In other words, denoting by $\psi:\mathbb{C}\backslash \{0\}\to (0, 2\pi)\times (0, \infty)$ the polar change of coordinates, 
the right expression for  \cite[Equation (1.1)]{LiWei2009} is:
\begin{equation}  
\frac{1}{\pi} \int_{T} \frac{1}{|z|^2} \frac{r_1^2 + r_2^2 - 2 r_{12}^2}{r_3^2 \sqrt{(r_1 + r_2)^2 - 4 r_{12}^2}} d\sigma(z)
=\frac{1}{\pi}\int_{\psi(T)} \frac{1}{\rho^2} \frac{r_1^2 + r_2^2 - 2 r_{12}^2}{r_3^2 \sqrt{(r_1 + r_2)^2 - 4 r_{12}^2}} \rho d\rho d\theta 
\end{equation}
This ambiguity is no longer present in their asymptotic analysis.

\begin{proof}
We follow closely the lines of the proof given in \cite[Thm. 1.1, Thm. 4.1]{LiWei2009},
adjusting certain computations as needed.
Also, we simplify the first part of their proof by not switching to polar coordinates
while obtaining equations (\ref{eq:Jac}) and (\ref{eq:quadform}) below.

Applying the Kac-Rice formula (restated in \cite[Lemma 2.1]{LiWei2009}), we have:
\begin{equation}\label{eq:KRlemma}
 \EE N_F(T) = \int_T \EE \left( | \det J_F(z)|  \big| F(z) = 0 \right) p(0;z) dA(z),
\end{equation}
where, for each $z$, $p(s;z)$ is the probability density function of the random variable $s=F(z)$.

The modulus of the Jacobian determinant of $F(z) = p(z) + \ol{q(z)}$ is given by (see \cite[Sec. 1.2]{Duren})
\begin{equation}\label{eq:Jac}
| J_F(z) | = \left| |p'(z)|^2 - |q'(z)|^2 \right| = \frac{1}{|z|^2} \left| |z p'(z)|^2 - |z q'(z)|^2 \right| ,
\end{equation}
and hence we have
\begin{equation}\label{eq:quadform}
 \EE \left( | \det J_F(z)|  \big| F(z) = 0 \right) = \frac{1}{|z|^2} \EE \left( |u_1^2 - u_2^2 + v_1^2 - v_2^2|  \big| u_3 = 0 , v_3 = 0 \right),
\end{equation}
where, for $j=1,2,3$, the expressions $u_j,v_j$ are given by
\begin{align*}
 u_1 &= \Re \sum_{k=0}^{n} k a_k z^{k}  , \quad v_1 = \Im\sum_{k=0}^n k a_k z^{k} , \\
 u_2 &= \Re \sum_{k=0}^{m} k b_k z^{k}  , \quad v_2 = \Im\sum_{k=0}^m k b_k z^{k} , \\
 u_3 &= \Re \left( p_n(z) + \ol{q_m(z)} \right) , \quad v_3 = \Im \left( p_n(z) + \ol{q_m(z)} \right).
\end{align*}

Then letting $(U_1,U_2,V_1,V_2)$ denote the Gaussian vector 
that has the distribution of $(u_1,u_2,v_1,v_2)$
under the (linear) condition $u_3=0,v_3=0$, we have
\begin{equation}\label{eq:smallbig}
 \EE \left( |u_1^2 - u_2^2 + v_1^2 - v_2^2|  \big| u_3 = 0 , v_3 = 0 \right) = \EE \left( |U_1^2 - U_2^2 + V_1^2 - V_2^2|  \right).
\end{equation}
The covariance matrix $R$ of $(U_1,U_2,V_1,V_2)$ is given by \cite[p. 30]{T90} 
\begin{equation}\label{eq:covR}
 R = C - B A^{-1} B^T ,
\end{equation}
where
\begin{align*}
 A_{2 \times 2} &= \text{cov}(u_3,v_3), \\
 B_{4 \times 2} &= \text{cov}((u_1,u_2,v_1,v_2),(u_3,v_3)), \\
 C_{4 \times 4} &= \text{cov}(u_1,u_2,v_1,v_2).
\end{align*}
First we compute 
\begin{align*}
 \EE u_3^2 = \EE v_3^2 &= \frac{1}{2} \sum_{k=0}^n \binom{n}{k} |z|^{2k} + \frac{1}{2} \sum_{k=0}^m \binom{n}{k} |z|^{2k} = \frac{1}{2}(1+|z|^2)^n + \frac{1}{2}P_{m,n}(|z|^2), \\
 \EE u_1 u_3 = \EE v_1 v_3 &= \frac{1}{2} \sum_{k=0}^n k \binom{n}{k} |z|^{2k} = \frac{1}{2} n |z|^2(1+|z|^2)^{n-1}, \\
 \EE u_1^2 = \EE v_1^2 &= \frac{1}{2} \sum_{k=0}^n k^2 \binom{n}{k} |z|^{2k} = \frac{1}{2} (n^2 |z|^4 + n |z|^2)(1+|z|^2)^{n-2}, \\
 \EE u_2 u_3 = \EE v_2 v_3 &= \frac{1}{2} \sum_{k=0}^m k \binom{n}{k} |z|^{2k} = \frac{1}{2} n |z|^2 P_{m-1,n-1}(|z|^2), \\
 \EE u_2^2 = \EE v_2^2 &= \frac{1}{2} \sum_{k=0}^m k^2 \binom{n}{k} |z|^{2k} = \frac{1}{2} (n^2 |z|^4 + n |z|^2) P_{m-2,n-2}(|z|^2), 
\end{align*}
and hence 
\begin{align*}
 A_{2 \times 2} &= \frac{(1+|z|^2)^n + P_{m,n}(|z|^2)}{2} \begin{pmatrix}
1 & 0  \\
 0 & 1 \\
\end{pmatrix}, \\
 B_{4 \times 2} &= \frac{1}{2} \begin{pmatrix}
 n|z|^2 (1+|z|^2)^{n-1} & 0  \\
 n |z|^2 P_{m-1,n-1}(|z|^2)& 0  \\
0 &  n |z|^2(1+|z|^2)^{n-1} \\
0 & n |z|^2 P_{m-1,n-1}(|z|^2)\\
\end{pmatrix}, \\
 C_{4 \times 4} &= \frac{1}{2} \text{diag} ( (n^2|z|^2 + n)(1+|z|^2)^{n-2} , (n^2|z|^2 + n)P_{m-2,n-2}(|z|^2) , \\
 & \quad \quad \quad \quad \quad (n^2|z|^2 + n)(1+|z|^2)^{n-2}, (n^2|z|^2 + n)P_{m-2,n-2}(|z|^2) ).
\end{align*}
From these we compute (\ref{eq:covR}):
$$ R_{4 \times 4} = \frac{1}{2 R_3} \begin{pmatrix}
 R_1 & -R_{12} & 0 & 0  \\
 -R_{12} & R_{2} & 0 & 0  \\
 0 & 0  & R_{1} & -R_{12}  \\
 0 & 0  & -R_{12} & R_{2} \\
\end{pmatrix}, $$
where
\begin{align*}
R_{12} &=n^2|z|^4(1+|z|^2)^{n-1} P_{m-1,n-1}(|z|^2), \\
R_3 &= (1+|z|^2)^n +P_{m,n}(|z|^2) , \\
R_1 &=  R_3 (n^2 |z|^4 + n |z|^2)(1+|z|^2)^{n-2} - n^2|z|^4(1+|z|^2)^{2n-2} , \\
R_2 &= R_3 [ n^2 |z|^4 P_{m-2,n-2}(|z|^2) + n|z|^2 P_{m-1,n-2}(|z|^2) ] - n^2|z|^4 [P_{m-1,n-1}(|z|^2)]^2  .
\end{align*}
Applying \cite[Cor. 2.1]{LiWei2009}, we obtain:
\begin{equation}\label{eq:applycor}
 \EE \left| U_1^2 - U_2^2 + V_1^2 - V_2^2 \right| = \frac{1}{R_3} \frac{R_1^2 + R_2^2 - 2 R_{12}^2}{\sqrt{(R_1+R_2)^2 - 4 R_{12}^2 }} .
\end{equation}

For each fixed $z$, the complex Gaussian $s=F(z)$ has 
probability density function
$$p(s;z) = \frac{1}{\pi R_3 } \exp \{ - |s|^2 / R_3 \},$$
and in particular
$$p(0;z) = \frac{1}{\pi R_3}.$$
Applying this along with equations (\ref{eq:quadform}), (\ref{eq:smallbig}), and (\ref{eq:applycor}) to
the Kac-Rice formula (\ref{eq:KRlemma}) we obtain the desired result (\ref{eq:KR}).
\end{proof}

\section{Proof of Theorem \ref{thm:asymp}}\label{sec:asymp}

\subsection{The case when $m = \alpha n$}

Applying Theorem \ref{thm:exact} with $N_F:=N_F(\CC)$, switching to polar coordinates $r = |z|$, $dA(z) = r dr d\theta$,
and integrating out the angular variable $\theta$, we are left with:
$$ \EE N_F = 2 \int_{0}^{\infty} \frac{1}{r} \frac{a_1^2 + a_2^2 - 2 a_{12}^2}{a_3^2 \sqrt{(a_1 + a_2)^2 - 4 a_{12}^2}} dr, $$
where
$$a_{12} = n^2r^4(1+r^2)^{n-1} P_{m-1,n-1}(r^2) ,$$
$$a_3 = (1+r^2)^n +P_{m,n}(r^2) ,$$
$$a_1 = a_3 (n^2r^4+nr^2) (1+r^2)^{n-2} - n^2r^4(1+r^2)^{2n-2} ,$$
$$a_2 = a_3 [ n^2 r^4 P_{m-2,n-2}(r^2) + nr^2 P_{m-1,n-2}(r^2) ] - n^2r^4 [P_{m-1,n-1}(r^2)]^2 .$$

Factoring $(1+r^2)^{4n-4}$ from the numerator and $(1+r^2)^{4n-2}$ from the denominator, we have:
\begin{equation}\label{eq:factored}
 \EE N_F = 2 n^{3/2} \int_{0}^{\infty} \frac{1}{n^{1/2} r (1+r^2)^2} \frac{b_1^2 + b_2^2 - 2 b_{12}^2}{b_3^2 \sqrt{(b_1 + b_2)^2 - 4 b_{12}^2}} dr,
\end{equation}
where
$$b_{12} = nr^4 \frac{P_{m-1,n-1}(r^2)}{(1+r^2)^{n-1}} ,$$
$$b_3 = 1 + \frac{P_{m,n}(r^2)}{(1+r^2)^n} ,$$
$$b_1 = b_3[nr^4 + r^2] - nr^4 ,$$
$$b_2 = b_3 \left[ n r^4 \frac{P_{m-2,n-2}(r^2)}{(1+r^2)^{n-2}} + r^2 \frac{P_{m-1,n-2}(r^2)}{(1+r^2)^{n-2}} \right] - nr^4 \left[\frac{P_{m-1,n-1}(r^2)}{(1+r^2)^{n-1}}\right]^2 .$$

We will apply Lebesgue's dominated convergence theorem to take the limit of the integral appearing in (\ref{eq:factored}).
The following claim implies that the sequence of integrands in (\ref{eq:factored}) is bounded by a single integrable function.

\noindent {\bf Claim:} 
$$  \frac{b_1^2 + b_2^2 - 2 b_{12}^2}{b_3^2 \sqrt{(b_1 + b_2)^2 - 4 b_{12}^2}} = O(\sqrt{n} r^3), \quad \text{ as } n \rightarrow \infty.$$
\begin{proof}[Proof of Claim]
First we note that $a_1 a_2 \geq a_{12}^2$.
This is by the Cauchy-Schwarz inequality, 
since it follows from the proof of Theorem \ref{thm:exact} that
$a_1 = \EE U_1^2, a_2 = \EE U_2^2,$ and $a_{12} = \EE U_1 U_2$,
where $U_1$ and $U_2$ are Gaussian random variables.

This implies that $b_1 b_2 \geq b_{12}^2$. Since $b_3 \geq 1$, we have:
\begin{align*}
 \frac{b_1^2 + b_2^2 - 2 b_{12}^2}{b_3^2 \sqrt{(b_1 + b_2)^2 - 4 b_{12}^2}} &\leq  \sqrt{b_1^2 + b_2^2 - 2 b_{12}^2} \\ 
 &= \sqrt{(b_1 - b_{2})^2 + 2 ( b_1 b_2 - b_{12}^2) } \\
 &\leq \sqrt{(b_1 - b_{2})^2 } + \sqrt{ 2(b_1 b_2 - b_{12}^2)} \\
 &= |b_1 - b_{2}| + \sqrt{2} \sqrt{ b_1 b_2 - b_{12}^2}
\end{align*}

Thus, it suffices to show that
\beq \label{eq:ists1}
b_1 - b_{2} = O(\sqrt{n} + r^2),
\eeq
and\
\beq \label{eq:ists2}
b_1 b_2 - b_{12}^2 = O( n r^6).
\eeq

Let $q_{m,n} := \frac{P_{m,n}(r^2)}{(1+r^2)^n}$.
Then, $b_3 = 1 + q_{m,n}$, and we have:
$$b_1 = (1+q_{m,n}) (nr^4+r^2)-nr^4,$$
and
$$b_2 = (1+q_{m,n}) \left( nr^4 q_{m-2,n-2} + r^2 q_{m-1,n-2} \right)-nr^4 q_{m-1,n-1}^2.$$
These lead to:
\begin{align*}
  b_1 - b_{2} = n r^4 &\left( (1 + q_{m,n} )(1 - q_{m-2,n-2}) - (1 - q_{m-1,n-1}^2) \right) \\
  &+ (1+q_{m,n}) r^2 \left( 1- q_{m-1,n-2} \right).
\end{align*}
The term $(1+q_{m,n}) r^2 \left( 1- q_{m-1,n-2} \right)$ is bounded by $2 r^2$.
We consider the remaining term $n r^4 \left( (1 + q_{m,n} )(1 - q_{m-2,n-2}) - (1 - q_{m-1,n-1}^2)  \right)$
which can be rewritten as:
\begin{align*}
  & n r^4 \left( (1 + q_{m,n} )(q_{m-1,n-1} - q_{m-2,n-2}) + (1 - q_{m-1,n-1}) (q_{m,n} - q_{m-1,n-1}) \right) \\
  \leq& n r^4 \left( 2 (q_{m-1,n-1} - q_{m-2,n-2}) + (q_{m,n} - q_{m-1,n-1}) \right)  \\
 =&   n r^4 \left( 2 \binom{n-2}{m-1} \frac{r^{2(m-1)}}{(1+ r^2)^{n-1}} + \binom{n-1}{m} \frac{r^{2m}}{(1+ r^2)^n} \right)  \\
 \leq&   n r^4 \left( 2 \binom{n-2}{m-1} + \binom{n-1}{m} \right) \frac{r^{2(m-1)}}{(1+ r^2)^{n-1}}   \\
 \leq&   3 n \binom{n-1}{m} \frac{r^{2(m+1)}}{(1+ r^2)^{n-1}} ,
\end{align*}
where we have used the identity 
\beq\label{eq:identity}
q_{m,n} - q_{m-1,n-1} = \binom{n-1}{m} \frac{r^{2m}}{(1+ r^2)^n} ,
\eeq
which can be seen as follows
\begin{align*}
q_{m,n} - q_{m-1,n-1} &= \frac{\sum_{k=0}^{m} \binom{n}{k} r^{2k} - (1+ r^2) \sum_{k=0}^{m-1} \binom{n-1}{k}r^{2k} }{(1+r^2)^n} \\
&= \frac{\sum_{k=1}^{m-1} \left( \binom{n}{k} - \binom{n-1}{k} - \binom{n-1}{k-1} \right) r^{2k} }{(1+r^2)^n} +  \binom{n-1}{m} \frac{r^{2m}}{(1+ r^2)^n} \\
&=  \binom{n-1}{m} \frac{r^{2m}}{(1+ r^2)^n}.
\end{align*}

Applying the first derivative test to $\frac{x^{m+1}}{(1+ x)^{n-1}}$ over the interval $x>0$ we find that the maximum occurs at $x=\frac{m+1}{n-m-2}$.
Thus, we have:
\begin{align*}
   3 n \binom{n-1}{m} \frac{r^{2(m+1)}}{(1+ r^2)^{n-1}}  &\leq   3 n \binom{n-1}{m} \frac{(\frac{m+1}{n-m-2})^{m+1}}{(\frac{n-1}{n-m-2})^{n-1}} \\
 &\leq   3 n \binom{n-1}{m} \frac{(m+1)^{m+1} (n-m-2)^{n-m-2}}{(n-1)^{n-1}} \\
 &=  3 n \frac{m+1}{n-1-m} \binom{n-1}{m+1} \frac{(m+1)^{m+1} (n-m-2)^{n-m-2}}{(n-1)^{n-1}}\\
 &\leq  C n \frac{m+1}{n-1-m} \sqrt{\frac{n-1}{(m+1)(n-m-2)}} = O(n^{1/2}),
\end{align*}
where we have used Stirling's approximation while
recalling that $m = \alpha n$.

This establishes (\ref{eq:ists1}).

\bigskip

Next we consider $b_1 b_2 - b_{12}^2$.

We have:
$$ b_1 b_2 - b_{12}^2 = n^2 r^8 \left( q_{m,n}\left[ (1+q_{m,n})q_{m-2,n-2} - q_{m-1,n-1}^2 \right] -q_{m-1,n-1}^2\right) $$
$$ + (1 + q_{m,n}) r^2 \left( b_2  + \frac{P_{m-1,n-2}(r^2)}{(1+r^2)^{n-2}} q_{m,n} \right).$$

Part of this can be estimated as follows:
$$(1 + q_{m,n}) r^2 \left( b_2  + \frac{P_{m-1,n-2}(r^2)}{(1+r^2)^{n-2}} q_{m,n} \right) \leq 2 r^2 (2(nr^4+ r^2) + 1) = O(nr^6).$$

Since $b_1 b_2 - b_{12}^2 \geq 0$, in order to prove (\ref{eq:ists2}), it is enough to show that for the remaining terms we have:
$$n^2 r^8 \left( q_{m,n}\left[ (1+q_{m,n})q_{m-2,n-2} - q_{m-1,n-1}^2 \right] -q_{m-1,n-1}^2\right) \leq 0.$$
We notice that
\begin{align*}  
 & \left( q_{m,n}\left[ (1+q_{m,n})q_{m-2,n-2} - q_{m-1,n-1}^2 \right] -q_{m-1,n-1}^2\right) \\
   =& \left( q_{m,n}(1+q_{m,n})q_{m-2,n-2} - (1+q_{m,n}) q_{m-1,n-1}^2 \right) \\ 
   =& \left(1+q_{m,n} \right) \left( q_{m,n}q_{m-2,n-2} - q_{m-1,n-1}^2 \right).
\end{align*}
We will show that $q_{m,n}q_{m-2,n-2} - q_{m-1,n-1}^2 \leq 0 $.

Using again the identity (\ref{eq:identity}), we have:
\begin{align*}
q_{m,n}q_{m-2,n-2} - q_{m-1,n-1}^2 &= q_{m,n} \left( q_{m-1,n-1} - \binom{n-2}{m-1} \frac{r^{2(m-1)}}{(1+r^2)^{n-1}} \right) - q_{m-1,n-1}^2  \\
&= q_{m-1,n-1}  \left( q_{m,n} - q_{m-1,n-1} \right)  - q_{m,n} \binom{n-2}{m-1} \frac{r^{2(m-1)}}{(1+r^2)^{n-1}}  \\
&= q_{m-1,n-1}  \binom{n-1}{m} \frac{r^{2m}}{(1+r^2)^{n}}   - q_{m,n} \binom{n-2}{m-1} \frac{r^{2(m-1)}}{(1+r^2)^{n-1}}  \\
&= \frac{\binom{n-2}{m-1} r^{2(m-1)}}{(1+r^2)^{2n-1}}  \left[ r^2 \frac{n-1}{m} P_{m-1,n-1}(r^2) - P_{m,n}(r^2) \right].
\end{align*}

Finally, we have:
$$ r^2 \frac{n-1}{m} P_{m-1,n-1}(r^2) - P_{m,n}(r^2) = -1 + \sum_{j=1}^{m} \left( \frac{n-1}{m} \binom{n-1}{j-1} - \binom{n}{j}\right) r^{2j} ,$$
and we see that each coefficient $\frac{n-1}{m} \binom{n-1}{j-1} - \binom{n}{j} = \binom{n-1}{j-1} \left( \frac{n-1}{m} - \frac{n}{j}  \right) $ is negative.

\end{proof}

Having justified an application of Lebesgue's dominated convergence theorem,
we find the pointwise limit of the integrand in (\ref{eq:factored}) using the following asymptotic 
(whose proof is given in Section \ref{sec:lemmaproof}).

\begin{lemma}\label{lemma:Laplace}
 Let $x \geq 0$.  For all $0<\alpha < 1$, we have (as $n \rightarrow \infty$ with $m=\alpha n$):
 $$\frac{P_{m,n}(x)}{(1+x)^n} = \left\{
\begin{array}{cc}
1+ O(1/n), \quad 0 \leq x < \frac{\alpha}{1-\alpha} ,\\
O(\exp\{-c n\}), \quad x > \frac{\alpha}{1-\alpha}. \\
\end{array} \right.
 $$
\end{lemma}

According to this asymptotic, for $r^2 > \frac{\alpha}{1-\alpha}$,
the integrand in Equation (\ref{eq:factored}) converges to zero,
and for $0 < r^2 < \frac{\alpha}{1-\alpha}$, we see that $b_2 = b_1 (1+O(1/n))$,
and
$$\frac{b_1^2 + b_2^2 - 2 b_{12}^2}{n^{1/2} b_3^2 \sqrt{(b_1 + b_2)^2 - 4 b_{12}^2}} = \frac{\sqrt{b_1^2 - b_{12}^2}}{n^{1/2} b_3^2} (1+O(1/n)) \sim \frac{r^3}{2} .$$

Thus, we have
$$ N_F \sim n^{3/2} \int_{0}^{\sqrt{\alpha/(1-\alpha)}} \frac{r^2}{(1+r^2)^2} dr = n^{3/2} c_\alpha,$$
where 
\begin{equation*}
c_\alpha =  \int_{0}^{ \sqrt{\alpha/(1-\alpha)} } \frac{r^2}{(1+r^2)^2} dr .
\end{equation*}
In order to determine $c_\alpha$, 
we make the change of variable $r = \tan (\theta), dr = \sec^2 (\theta) d \theta$:
\begin{equation*}
\int_{0}^{ \sqrt{\alpha/(1-\alpha)} } \frac{r^2}{(1+r^2)^2} dr = \int_{0}^{A} \sin^2(\theta) d\theta, \quad A = \arctan \left( \sqrt{\frac{\alpha}{1-\alpha}} \right).
\end{equation*}
Thus, we have
\begin{equation*}
c_\alpha = \frac{1}{2} \left( \arctan \left( \sqrt{ \frac{\alpha}{1-\alpha} }\right) - \sqrt{\alpha(1-\alpha)} \right).
\end{equation*}
This completes the proof of Theorem \ref{thm:asymp} in the case that $m = \alpha n$ with $0<\alpha<1$.

\subsection{The case when $n \rightarrow \infty$ with $m$ fixed}

This case is simpler and does not require Lemma \ref{lemma:Laplace}.
Omitting the details, 
we find that $b_2, b_{12}$, converge to zero, $b_3$ converges to $1$,
and
$$ N_F \sim 2 n \int_{0}^{\infty} \frac{r}{(1+r^2)^2} dr = n. $$

\subsection{Asymptotics of the Kac-Rice density}\label{sec:crit}

Consider again the case when $m = \alpha n$, with $0< \alpha < 1$.
Above, we have factored out $n^{3/2}$ from the Kac-Rice density in order to apply the dominated convergence theorem,
but Lemma \ref{lemma:Laplace} can also be used to find the pointwise asymptotic.
The Kac-Rice density is asymptotic (as $n \rightarrow \infty$) to $\frac{n^{3/2} |z|}{2\pi (1+|z|^2)^2} $
for $|z| < \sqrt{ \alpha/(1-\alpha) }$,
and it is asymptotic to $\frac{n}{\pi (1+|z|^2)^2} $ for $|z| > \sqrt{ \alpha/(1-\alpha) }$.
Thus, the leading contribution of zeros are located within the distance $\sqrt{\alpha/(1-\alpha) }$ from the origin.
This critical radius originates in the proof of Lemma \ref{lemma:Laplace}
(based on Laplace's method), 
see Cases 1 and 2 in Section \ref{sec:lemmaproof}.

\section{Proof of Lemma \ref{lemma:Laplace} using Laplace's method}\label{sec:lemmaproof}

The following formula is provided in \cite[Lemma 1]{Ostr2004}.

\begin{lemma}\label{lemma:Pform}
For $0<m<n-1$
 \begin{equation}\label{eq:asympt}\frac{P_{m,n}(x)}{(1+x)^n} = \binom{n}{m}(n-m) \int_{x/(x+1)}^1 u^m(1-u)^{n-m-1} du .\end{equation}
\end{lemma}

We apply Laplace's method to derive Lemma \ref{lemma:Laplace} 
from Lemma \ref{lemma:Pform}.
Rewriting the integrand, we have for $m=\alpha n$:
 $$\int_{x/(x+1)}^1 e^{ n h(u) } g(u) du,$$
 where $h(u) = [ \alpha \log (u) + (1-\alpha) \log (1-u) ]$,
 and $g(u) = (1-u)^{-1}$.

 \medskip
 
 \noindent {\bf Case 1:} When $x/(x+1) < \alpha$, $h(u)$ achieves its maximum at 
$u=\alpha$,
the unique solution of the saddle-point equation:
$$h'(u) = \alpha/u - (1-\alpha)/(1-u) = 0.$$
Applying Laplace's method \cite[Sec. 3.4]{Miller}, we have:
\begin{align*}
 \int_{x/(x+1)}^1 e^{ n h(u) } g(u) du &= e^{ n h(\alpha) } g(\alpha) \sqrt{\frac{2 \pi}{-n h''(\alpha)}} \left( 1 + O(n^{-1}) \right) \\
 &= \alpha^{\alpha n} (1-\alpha)^{(1-\alpha)n-1} \sqrt{\frac{2 \pi \alpha(1-\alpha)}{n}} \left( 1 + O(n^{-1}) \right).
\end{align*}
Applying Stirling's approximation, we have:
\begin{equation}\label{eq:stirling}
 \binom{n}{m}(n-m) = \frac{\sqrt{n}}{\alpha^{\alpha n}(1-\alpha)^{(1-\alpha)n-1}\sqrt{2 \pi \alpha (1-\alpha)}} \left( 1 + O(n^{-1}) \right).
 \end{equation}
Combining these results into \eqref{eq:asympt}, we find:
 $$\frac{P_{m,n}(x)}{(1+x)^n} = \left( 1 + O(n^{-1}) \right).$$ 

 \noindent {\bf Case 2:} When $x/(x+1) > \alpha$, 
 the saddle-point $u = \alpha$ is outside of the interval of integration,
 and $h(u)$ instead achieves its maximum at the left end-point
 $u = x/(x+1)$.  We thus have (by the alternative form of Laplace's method \cite[Sec. 3.3]{Miller}):
\begin{align*}
 \int_{x/(x+1)}^1 e^{ n h(u) } g(u) du &= e^{ n h(x/(x+1)) } \frac{g(x/(x+1))}{-nh'(x/(x+1))} \left( 1 + O(n^{-1}) \right) \\
 &\sim \left( \frac{x}{x+1} \right)^{\alpha n+1} \left( \frac{1}{x+1} \right)^{(1-\alpha)n-1} \left( \frac{1}{n ( x(1-\alpha) - \alpha)}  \right).
\end{align*}
Combining this with (\ref{eq:stirling}),
we see that 
 $$\frac{P_{m,n}(x)}{(1+x)^n} \sim c_1(x,\alpha) \frac{e^{-c_2(x,\alpha)n}}{\sqrt{n}}.$$ 

\section{Concluding remarks}

We have shown that the average number of zeros of a 
random harmonic polynomial sampled from the truncated model
has order $n^{3/2}$ when $m$ is a fixed fraction of $n$
and grows linearly in $n$ when $m$ is fixed.
In comparison with the Li-Wei model \cite[Thm. 1.1]{LiWei2009},
this behavior seems more indicative of (a probabilistic version) of Wilmshurst's conjecture.

Extending the above-mentioned breakthrough \cite{K-S},
Khavinson and Neumann \cite{K-N} used anti-holomorphic dynamics 
to count zeros of rational harmonic functions of the form $r(z) + \bar{z}$,
giving a complete solution to astronomer S-H. Rhie's conjecture \cite{Rhie} in gravitational lensing.
For further discussion and related results, see \cite{K-N2, K-L, BHJR}.
In order to model stochastic gravitational lensing,
the zeros of random harmonic functions were 
studied by A. Wei in his thesis \cite[Ch. 3]{Wei}
and by Petters, Rider, and Teguia \cite{PRT1, PRT2}.


Is the variance of $N_F$ asymptotically proportional to the mean?
Computer experiments in \cite{HLLM} suggest that the answer is (perhaps surprisingly, cf. \cite{GranWig2011}) ``no'',
and that the variance instead has order $n^2$.




\begin{thebibliography}{}


\bibitem{BHJR}
P. M. Bleher, Y. Homma, L. L. Ji, R. K. W. Roeder,
\emph{Counting zeros of harmonic rational functions and its application to gravitational lensing}, Internat. Math. Res. Notices., 8 (2014), 2245-2264.

\bibitem{BHS1995}
D. Bshouty, W. Hengartner, and T. Suez, \emph{The exact bound on the number of zeros of harmonic polynomials}, J. Anal. Math., 67 (1995), 207-218.

\bibitem{BL2004}
D. Bshouty, A. Lyzzaik, \emph{On Crofoot-Sarason's conjecture for harmonic polynomials}, Comp. Meth. Funct. Thy.,
4 (2004), 35-41.

\bibitem{BL2010}
D. Bshouty, A. Lyzzaik, \emph{Problems and conjectures for planar harmonic mappings: in
the Proceedings of the ICM2010 Satellite Conference: International Workshop on Harmonic and Quasiconformal Mappings (HQM2010)},
Special issue in: J. Analysis, 18 (2010), 69-82.

\bibitem{DHL}
P. Duren, W. Hengartner, R. S. Langesen, \emph{The argument principle for harmonic functions}, Amer. Math. Monthly, 103 (1996), 411-415.

\bibitem{Duren}
P. Duren, \emph{Harmonic Mappings in the Plane}, Cambridge University Press, 2004.

\bibitem{EK1995}
A. Edelman and E. Kostlan, \emph{How many zeros of a random polynomial are real?}, Bull. Amer. Math. Soc. 32 (1995), 1-37.

\bibitem{G2008}
L. Geyer, \emph{Sharp bounds for the valence of certain harmonic polynomials}, Proc. AMS, 136 (2008), 549-555.

\bibitem{GranWig2011}
A. Granville, I. Wigman, \emph{The distribution of the zeros of random trigonometric polynomials}, Amer. J. Math., 133 (2011), 295-357.

\bibitem{HLLM}
J. Hauenstein, A. Lerario, E. Lundberg, D. Mehta, 
\emph{Experiments on the zeros of harmonic polynomials using certified counting}, Exper. Math., 24 (2015), 133-141.

\bibitem{K-L}
D. Khavinson, E. Lundberg, \emph{Transcendental harmonic mappings and gravitational lensing by isothermal galaxies}, Complex Anal. Oper. Thy. 4 (2010), 515-524.

\bibitem{K-N}
D. Khavinson, G. Neumann, \emph{On the number of zeros of certain rational harmonic functions}, Proc. AMS, 134 (2006), 1077-1085.

\bibitem{K-N2}
D. Khavinson, G. Neumann, \emph{From the fundamental theorem of algebra to astrophysics: a harmonious path}, Notices AMS, 55 (2008), 666-675.

\bibitem{K-S}
D. Khavinson, G. Swiatek, \emph{On a maximal number of zeros of certain harmonic polynomials}, Proc. AMS, 131 (2003), 409-414.

\bibitem{LLL}
S-Y. Lee, A. Lerario, E. Lundberg, \emph{Remarks on Wilmshurst's theorem}, Indiana U. Math. J., 64 No. 4 (2015), 1153–1167

\bibitem{LLrandlem}
A. Lerario, E. Lundberg, \emph{On the geometry of random lemniscates}, preprint.

\bibitem{LiWei2009}
W. V. Li, A. Wei, \emph{On the expected number of zeros of random harmonic polynomials}, Proc. AMS, 137 (2009), 195-204.

\bibitem{LPX}
D. S. Lubinsky, I. E. Pritsker, X. Xie, \emph{Expected number of real zeros for random linear combinations of orthogonal polynomials}, preprint: http://arxiv.org/abs/1503.06376

\bibitem{Miller}
P. Miller, \emph{Applied Asymptotic Analysis}, Amer. Math. Soc., 2006.

\bibitem{Ostr2004}
I. V. Ostrovskii, \emph{On a problem of A. Eremenko}, CMFT, 4 (2004), 275-282.

\bibitem{PSch1996}
R. Peretz, J. Schmid, \emph{On the zero sets of certain complex polynomials}, Proceedings of the Ashkelon Workshop on Complex Function Theory (1996), 203-208,
Israel Math. Conf. Proc. 11, Bar-Ilan Univ. Ramat Gan, 1997.

\bibitem{PRT1}
A.O. Petters, B. Rider, A.M. Teguia, \emph{A mathematical theory of stochastic microlensing I: random time-delay functions and lensing maps}, J. Math. Phys. 50 (2009), 072503.

\bibitem{PRT2}
A.O. Petters, B. Rider, A.M. Teguia, \emph{A mathematical theory of stochastic microlensing II: random images, shear, and the Kac-Rice Formula}, J. Math. Phys. 50 (2009), 122501.

\bibitem{Rhie}
S. Rhie, \emph{n-point gravitational lenses with 5(n-1) images}, archiv:astro-ph/0305166 (2003).

\bibitem{Sh2002}
T. Sheil-Small, \emph{Complex Polynomials}, Cambridge University Press, 2002.

\bibitem{T90}
Y.L. Tong, \emph{The multivariate normal distribution}, Springer-Verlag, 1990.

\bibitem{Wei}
A. Wei, \emph{Random harmonic functions and multivariate Gaussian estimates}, D. Phil. thesis, Univ. of Delaware, 2009.

\bibitem{Wilm94}
A. S. Wilmshurst, \emph{Complex harmonic polynomials and the valence of harmonic polynomials}, D. Phil. thesis, Univ. of York, U.K., 1994.

\bibitem{Wilm98}
A. S. Wilmshurst, \emph{The valence of harmonic polynomials}, Proc. AMS 126 (1998), 2077-2081.

\bibitem{ZZ}
O. Zeitouni, S. Zelditch, \emph{Large deviations of empirical measures of zeros of random polynomials}, IMRN, Vol. 2010, No. 20, 3935-3992.


\end{thebibliography}
\end{document}